\documentclass{amsart}
\usepackage{amsmath}
\usepackage{amssymb}
\usepackage[all]{xy}
\usepackage{graphicx}
\usepackage{mathrsfs}

\numberwithin{equation}{section}

\usepackage{url}

\usepackage[latin1]{inputenc}
\usepackage{xspace,amssymb,amsfonts,euscript}
\usepackage{amsthm,amsmath}
\usepackage{palatino}
\usepackage{euscript}
\input xy \xyoption {all}

\usepackage{tikz}

\RequirePackage{color}
\definecolor{myred}{rgb}{0.75,0,0}
\definecolor{mygreen}{rgb}{0,0.5,0}
\definecolor{myblue}{rgb}{0,0,0.65}

\RequirePackage{ifpdf}
\ifpdf
  \IfFileExists{pdfsync.sty}{\RequirePackage{pdfsync}}{}
  \RequirePackage[pdftex,
   colorlinks = true,
   urlcolor = myblue, 
   citecolor = mygreen, 
   linkcolor = myred, 
   pagebackref,
   bookmarksopen=true]{hyperref}
\else
  \RequirePackage[hypertex]{hyperref}
\fi

\RequirePackage{ae, aecompl, aeguill} 

    \def\CM{{\mathbb{C}}}

    \def\FM{{\mathbb{F}}}

    \def\KM{{\mathbb{K}}}

    \def\NM{{\mathbb{N}}}
    \def\OM{{\mathbb{O}}}
    
    \def\QM{{\mathbb{Q}}}
    \def\RM{{\mathbb{R}}}

    \def\ZM{{\mathbb{Z}}}

    \def\EC{{\mathcal{E}}}
    \def\FC{{\mathcal{F}}}

    \def\LC{{\mathcal{L}}}

\def\a{\alpha}

\def\G{\Gamma}

\def\l{\lambda}
\def\L{\Lambda}

\newcommand{\nc}{\newcommand} \newcommand{\renc}{\renewcommand}

\newcommand{\rdots}{\mathinner{ \mkern1mu\raise1pt\hbox{.}
    \mkern2mu\raise4pt\hbox{.}
    \mkern2mu\raise7pt\vbox{\kern7pt\hbox{.}}\mkern1mu}}

\def\ov{\overline}
\def\un{\underline}

\def\to{\rightarrow}
\def\from{\leftarrow}
\def\longto{\longrightarrow}

\nc{\triright}{\stackrel{[1]}{\to}}
\nc{\longtriright}{\stackrel{[1]}{\longto}}

\nc{\Hb}{H^\bullet}

\nc{\Br}{\mathcal{B}}
\nc{\HotRR}{{}_R\mathcal{K}_R}
\nc{\HotR}{\mathcal{K}_R}
\nc{\excise}[1]{}
\nc{\defect}{\text{df}}
\nc{\h}[1]{\underline{H}_{#1}}

\nc{\Ga}{\mathbb{G}_a} 
\nc{\Gm}{\mathbb{G}_m} 

\nc{\Perv}{{\mathbf{P}}}

\nc{\IH}{{\mathrm{IH}}}

\nc{\ic}{\mathbf{IC}}

\nc{\gl}{{\mathfrak{gl}}}
\renc{\sl}{{\mathfrak{sl}}}
\renc{\sp}{{\mathfrak{sp}}}

\renc{\Im}{\textrm{Im}}

\nc{\HBM}{H^{BM}}

 \DeclareMathOperator{\Hom}{Hom}
\DeclareMathOperator{\supp}{supp} 
\DeclareMathOperator{\End}{End} 
\DeclareMathOperator{\Proj}{Proj}
\DeclareMathOperator{\proj}{proj}
\DeclareMathOperator{\Rep}{Rep}

\DeclareMathOperator{\rank}{rank}
\DeclareMathOperator{\Seq}{Seq}

\newtheorem{thm}{Theorem}[section]
\newtheorem{lem}[thm]{Lemma}

\newtheorem{prop}[thm]{Proposition}

\newtheorem{conj}[thm]{Conjecture}

\theoremstyle{definition}

\newtheorem{Eg}[thm]{Example}

\theoremstyle{remark}
\newtheorem{remark}[thm]{Remark}

\DeclareMathOperator{\Ext}{Ext}

\newcommand{\into}{\hookrightarrow}

\nc{\simto}{\stackrel{\sim}{\to}}

\DeclareMathOperator{\gdim}{\underline{dim}}

\title{On an analogue of the James conjecture}

\author{Geordie Williamson}
\address{Max-Planck-Institut f\"ur Mathematik, Vivatsgasse 7, 53111 Bonn, Germany}
\email{geordie@mpim-bonn.mpg.de}
\urladdr{http://people.mpim-bonn.mpg.de/geordie/}

\begin{document}

\begin{abstract}We give a counterexample to the most optimistic
  analogue (due to Kleshchev and Ram) of the James conjecture for
  Khovanov-Lauda-Rouquier algebras associated to
  simply-laced Dynkin diagrams. The first counterexample occurs in type
  $A_5$ for $p = 2$ and involves
  the same singularity used by Kashiwara and Saito to show the
  reducibility of the characteristic variety of an intersection
  cohomology $D$-module on a quiver variety. Using recent results of Polo one
  can give counterexamples in type $A$ in all characteristics.
\end{abstract}

\maketitle

\begin{center}
  \emph{Dedicated to Jimi}
\end{center}

\section{Introduction}

A basic question in representation theory asks for the dimensions of the simple
modules for the symmetric group over an arbitrary field. Despite over
a century's effort even a conjectural formula remains out of reach.

Any simple representation in characteristic zero may be represented by
integral matrices, and one obtains a representation in
characteristic $p$ by reducing these matrices modulo $p$. This induces a well-defined
map on Grothendieck groups, and the matrix of this map (in the basis
of simple modules) is the ``decomposition matrix''. Knowledge of the
decomposition matrices for the symmetric group would give a formula
for the dimensions of the simple modules in characteristic $p$.

The group algebra of the symmetric group has a deformation given by
the Hecke algebra, and if one specialises the
quantum parameter to a $p^{th}$-root of unity one obtains an
algebra in characteristic zero whose representation theory is 
similar to that of the symmetric group in characteristic
$p$. The Hecke algebra allows one to factor the decomposition
matrix as the product of two matrices: the first controls
specialising the quantum parameter (a characteristic zero question),
and the second controls reduction modulo $p$. Moreover, in 1990 James \cite{J}
conjectured that this second ``adjustment'' matrix should (under
certain explicit lower bounds on the characteristic) be trivial.\footnote{After
  this article was submitted the author discovered a way of producing
  counterexamples to Lusztig's conjecture. These
  results may be used to deduce counterexamples to the James
  conjecture \cite{W}.}

The first matrix (controlling the characters of simple modules for the
Hecke algebra at a $p^{th}$ root of unity) may be calculated using
Schur-Weyl duality and Lusztig's character formula for quantum groups
at a root of unity \cite{L0}. More recently, Lascoux,
Leclerc and Thibon \cite{LLT} formulated conjectures giving a 
faster algorithm to calculate these 
matrices in terms of the dual canonical basis of Fock space for the
quantum affine special linear group. These conjectures were
subsequently proved by Ariki \cite{A}. (A proof was also announced by
Grojnowski.) Hence the Hecke algebra problem has been solved and one
would like to know when the James conjecture is valid.

Over the last decade these ideas have been gathered under the umbrella
of categorification and the important role played by
Khovanov-Lauda-Rouquier (KLR) algebras \cite{KL,2km} has become clear. This is
mainly thanks to the Brundan-Kleshchev isomorphism \cite{BK}: the
group algebras of symmetric groups, as well as Hecke algebras at roots
of unity are isomorphic to cyclotomic KLR algebras. It follows that the
representation theory of symmetric groups and Hecke algebras acquires a
non-trivial grading.  (The reader is
referred to \cite{K} for a survey of these and related
ideas.)

The Brundan-Kleshchev isomorphism identifies the symmetric group algebra in
characteristic $p$ and Hecke algebras at a $p^{th}$ root of unity with
KLR algebras associated to a cyclic Dynkin quiver with $p$ nodes
(considered over a field of characteristic $p$ and $0$
respectively). Using this isomorphism the James conjecture may be translated into the natural
statement that the decomposition numbers for these KLR algebras are
trivial, under certain explicit lower bounds on $p$.

The upshot is that KLR algebras associated to cyclic quivers could
help us make progress on one of the oldest questions in representation theory.
This begs the question: what do KLR algebras associated to Dynkin
quivers mean? One might hope that their modular representation theory
is simpler, and that we can garner a clue as to how
to approach the cyclic case. The following conjecture was proposed
by Kleshchev and Ram \cite[Conjecture 7.3]{KR} as a finite analogue of the James conjecture:

\begin{conj} \label{conj:KR}
  The decomposition numbers for KLR algebras associated to Dynkin
  quivers are trivial.
\end{conj}

The purpose of this note is to explain what this conjecture means
geometrically, and give a counterexample. Over a field of
characteristic zero it is a theorem of 
Varagnolo-Vasserot \cite{VV} and Rouquier \cite{R} that KLR algebras
are the Ext algebras of Lusztig sheaves on the moduli space of quiver
representations. Hence understanding projective modules over KLR
algebras is the same as understanding the decomposition of Lusztig
sheaves with coefficients of characteristic zero. The Decomposition
Theorem \cite{BBD} allows one
to translate this into a combinatorial problem. Maksimau \cite{M} has
recently extended this result: KLR algebras are the Ext algebras
of Lusztig sheaves over $\ZM$.  Here the decomposition theorem is
missing, but Kleshchev and Ram's conjecture is equivalent to the statement
that the decomposition of Lusztig sheaves should remain the same with
coefficients in any field.

This in turn is equivalent to asking that the stalks
and costalks of intersection cohomology complexes on moduli spaces of
quiver representations be free of torsion. (Or equivalently, that parity
sheaves (see \S \ref{sec:parity-sheav-quiv}) are isomorphic to intersection cohomology complexes.) A number
of people had hopes that similar statements would be true on flag
varieties in type $A$. In fact, the statement for moduli spaces of
quiver representations would imply the corresponding statement for
flag varieties, as we will see.\footnote{It is possible that these two
  statements are equivalent: see the two sentences after Problem 1 in \cite[\S 6.2]{KS}.}

In 2004 Braden gave counterexamples to these hopes on the flag
variety of $GL_8$ \cite{WB}. In this note we explain the relation between these
questions and give the first counterexample to Conjecture
\ref{conj:KR}. It occurs for the
Dynkin quiver of type $A_5$ with coefficients of characteristic 2. We
also explain how one can use recent results of Polo to give
counterexamples for type $A$ Dynkin quivers in all
characteristics. Unfortunately, these examples are large: to get a
counterexample in characteristic $p$ using the approach that we
discuss here one needs to consider a KLR
algebra associate to a Dykin quiver of type $A_{8p-1}$ and dimension
vector $(1, 2, \dots, 4p, \dots, 2,1)$. Hence the prospects of simple-minded
attempts to use KLR algebras to actually understand what is going on seem a little
dim! 

Finally, the first counterexample (on the $A_5$
quiver in characteristic $2$) is related to a certain subvariety in the moduli space of
quiver representations which was used by Kashiwara and Saito to give a
counterexample to a conjecture of Kazhdan and Lusztig on the
irreducibility of the characteristic cycle \cite{KS}. In fact, the examples that
we discuss in this paper all have reducible characteristic cycles, by
a result of Vilonen and the author \cite{VW}. Hence
Kleshchev and Ram's conjecture would have been implied by
Kazhdan and Lusztig's, had it been correct.

\emph{Acknowledgements:} I would like to thank Joe Chuang and Andrew
Mathas for explaining Kleshchev and Ram's conjecture to me,
Alexander Kleshchev, Ruslan Maksimau and Ulrich Thiel for useful
  discussions, and Carl Mautner for comments on a previous
  version. I am very grateful to Jon Brundan and Alexander
  Kleshchev for finally giving algebraic confirmation of the
  counterexample. Finally, thanks to the referee for a
  careful reading and useful comments.

\section{Brauer reciprocity for graded rings}

In the modular representation theory of finite groups an important
role is played by Brauer reciprocity. Brauer reciprocity is
an equivalence between the two problems of determining the composition
factors of the reduction modulo $p$ of a simple module in
characteristic 0, and decomposing
the lift of a projective module in characteristic $p$ to
characteristic 0. It will be important for us because it rephrases
the question of decomposition numbers for simple modules over KLR
algebras into questions about projective modules. As we will see,
the question of decomposing lifts of projective modules to
characteristic 0 has a
straightforward geometric interpretation.

\subsection{Graded algebras and modules}
\label{sec:gam}
Throughout we fix a prime number $p$ and let $\FM = \FM_p$, $\OM =
\ZM_p$ and $\KM = \QM_p$ (the finite field with $p$ elements, the
$p$-adic integers, and the $p$-adic numbers respectively). All of the
results of this paper remain valid for any 
$p$-modular system $(\KM,\OM,\FM)$.


Now let $H_\OM = \bigoplus_{i \in \ZM} H_i$ denote a finitely
generated graded $\OM$-algebra, such that $H^i$ is a finitely
generated free $\OM$-module for all $i \in \ZM$. We set $H_\FM := \FM \otimes_\OM H$
and $H_\KM := \KM \otimes_\OM H$. These are also finitely generated
graded algebras, finite dimensional in each degree. We assume in
addition that there exists a graded polynomial subring $A_\OM = \OM[X_1,
\dots, X_m] \subset H_\OM$  with generators of positive degree such
that:
\begin{enumerate}
\item $A_\OM$ is contained in the centre of $H_\OM$;
\item $A_\OM$ is a summand of $H_\OM$ as an $\OM$-module;
\item $H_\OM$ is a finitely generated module over $A_{\OM}$.
\end{enumerate}

For $k \in \{ \FM, \KM\}$ we set $A_k := A_{\OM} \otimes_\OM k$ which is
also a graded polynomial ring contained in the centre of $H_k$. Write
$A_k^+$ for the ideal of polynomials of positive degree.
Any graded simple $H_k$-module is annihilated by $A_k^+$ (by the graded
Nakayama lemma) and hence is a module over $H_k/(A_k^+)H_k$, a finite
dimensional graded algebra by our assumption (3). We conclude that
$H_k$ has finitely many graded simple modules up to shifts, all of
which are finite dimensional. 

For $k \in \{ \KM, \OM, \RM \}$ let $\Rep H_k$ denote the 
abelian category of all finitely generated graded
$H_k$-modules. That is, objects of 
$\Rep H_k$ are graded left $H_k$-modules, and morphisms are homogenous
of degree zero. Let $\Rep^f 
H_k$ denote the full subcategory of modules which are
finite dimensional over $k$, and let $\Proj H_k$ denote the full
subcategory of projective modules which are finitely generated over
$H_k$. Given $M = \bigoplus M^i \in \Rep H_k$ let $M(j)$ denote its shift, given
by $M(j)^i := M^{j+i}$. Given $M, N \in \Rep H_k$ set
$\Hom^\bullet(M,N) := \bigoplus \Hom(M,N(i))$. In this paper a 
\emph{graded category} will always mean a category equipped with a
self-equivalence (see the end of \cite[\S2.2.1]{2km} for a discussion).
We view $\Rep H_k$ as a graded category with respect
to the self-equivalence $M \mapsto M(1)$.

\subsection{Grothendieck groups}
 The Grothendieck groups of $\Rep^f H_k$ and $\Proj H_k$ will be
denoted by $[\Rep^f H_k]$ and $[\Proj H_k]$ respectively. They are
naturally $\ZM[v^{\pm 1}]$-modules by declaring $v [M]
:= [M(1)]$. Given $P \in \Proj H_k$ and $M \in \Rep^f H_k$ we obtain a pairing
$\langle -, - \rangle : [\Proj H_k] \times [\Rep^f H_k] \to \ZM[v^{\pm
  1}]$ by setting
\[
\langle [P], [M] \rangle := \gdim \Hom^\bullet(P,M),
\]
here $\gdim V := \dim V^iv^{-i} \in
\NM[v^{\pm 1}]$ for a finite dimensional graded vector space $V$.

It is known \cite[Corollary 9.6.7, Theorem 9.6.8]{NO} that every
simple module over
$H_k/(A_k^+)H_k$ admits a grading,\footnote{Here is an intuitive
  explanation of this fact, which was explained to me by Ivan Losev.
 Let $C$ be a finite-dimensional graded algebra. The grading on $C$ is
 equivalent to an action of the multiplicative group on $C$. A module is gradable if and only if its twist by any element
 of the multiplicative group yields an isomorphic module. Now twisting
 preserves simple modules and their moduli are
 discrete (because $C$ is finite-dimensional). Hence they are fixed, and hence gradable.
} and this grading is unique up to
isomorphism and shifts. By the Krull-Schmidt property every graded
simple modules $L$ admits a projective cover $P_L$. If we fix a choice
$L_1, \dots, L_m$ of graded simple module and let $P_1, \dots, P_m$
be their projective covers then the classes $[L_1], \dots, [L_m]$
(resp. $[P_1], \dots, [P_m]$ give a free $\ZM[v^{\pm 1}]$-basis for
$[\Rep H_k]$ (resp. $[\Proj H_k]$). Moreover, these bases are dual
under $\langle -, - \rangle$.

\subsection{Decomposition maps}
Given any $M \in \Rep^f H_\KM$ one can always find an
$H_\OM$-stable lattice $M_\OM \subset M$ by applying $H_{\OM}$ to a set of
generators for $M$. One can argue as in \cite[Part III]{S} that the class $[\FM
\otimes_\OM M_\OM]$ in the Grothendieck group of $\Rep^f H_{\FM}$
does not depend on the choice of lattice. In this way one obtains the
decomposition map between the Grothendieck groups:
\[
d : [\Rep^f(H_\KM)] \to [\Rep^f(H_\FM)].
\]
One the other hand, standard arguments (e.g. \cite[Theorem 12.3]{F})
show that idempotents in $H_\FM$
lift to $H_\OM$. It follows that given any projective module $P_\FM$
over $H_\FM$ there exists a projective module $P_\OM$ over $H_\OM$ such that $\FM
\otimes_\OM P_\OM \cong P_\FM$. Moreover, $P_\OM$ is unique up to (non-unique)
isomorphism by Nakayama's lemma. This process gives us the extension map:
\begin{gather*}
e : [\Proj(H_\FM)] \to [\Proj(H_\KM)] \\
[P_\FM] \mapsto [\KM \otimes_\OM P_\OM]
\end{gather*}

\subsection{Brauer reciprocity}
One way of phrasing Brauer reciprocity is that $d$ and $e$ are adjoint
with respect to the canonical pairing (the proof is identical to
\cite[Part III]{S}):

\begin{lem} \label{lem:BR}
We have
\[
\langle e([P]), [L] \rangle = \langle [P], d([L]) \rangle.
\]
for all $P \in \Proj H_\FM$ and $L \in \Rep^f H_\KM$.
\end{lem}

Equivalently the matrices (with entries in $\ZM[v^{\pm 1}]$) for $d$ and $e$ in the basis
$\{ [L_i] \}_{i = 1}^m$ and $\{ [P_i] \}_{i =1}^m$ are transposes of one another.

\section{Quiver varieties and parity sheaves}

In this section we recall briefly the definition of moduli spaces of quiver
representations and explain why it makes sense to study parity sheaves
on these spaces. Using a result of Maksimau \cite{M} we then explain why describing the characters of
parity sheaves on these spaces is equivalent to describing the
indecomposable projective modules over Khovanov-Lauda-Rouquier
algebras.

With the (possible) exception of \S\ref{sec:parity-sheav-quiv} and
\S\ref{sec:parity-sheav-kleshch} the material in this section is
standard. The material concerning constructible sheaves on moduli of
quiver representations is due to Lusztig \cite{L1,L2} (see \cite{Sch}
for an excellent survey). For the relation to KLR algebras see
\cite{VV,R,M}.

\subsection{Moduli of quiver representations}

Let $\G$ denote a quiver with vertex set $I$. Recall that a
representation of $\G$ is an $I$-graded vector space $V = \oplus V_i$  together with
linear maps $V_i \to V_j$ for each arrow $i \to j$ of $Q$. Let $\Rep Q$ denote
the abelian category of complex representations of $\G$. A \emph{dimension
  vector} is an element of the monoid $\NM[I]$ of formal $\NM$-linear
combination of the elements of $I$. Given an $I$-graded vector space $V$ its dimension vector is
$\gdim {V} = \sum (\dim V_i) i \in \NM[I]$.

%


Fix a dimension vector $d = \sum d_i i \in \NM[I]$ and a complex
 $I$-graded vector space $V$ with $\dim V = d$. Consider the space
\[
E_V := \prod_{ i \to j} \Hom(V_i, V_j)
\]
where the product is over all arrows $i \to j$ in $Q$. Then
$E_V$ is the space of all representations of the quiver $Q$ of
dimension vector $d$ together with fixed isomorphism with $V_i$ at
each vertex $i \in I$. If we let
\[
G_V := \prod_{i \in I} GL(V_i)
\]
be the group of grading preserving automorphisms of $V$ 
then $G_V$ acts on $E_V$ by $(g_i) \cdot (\alpha_{i\to j}) := (g_j
\alpha_{i\to j} g_i^{-1})$. The points of the quotient space
$G_d \setminus E_V$ correspond to isomorphism classes of
representations of $Q$ of dimension vector $d$.


\subsection{Flags and proper maps}
Let us fix a dimension vector $d
\in \NM[I]$ and a sequence ${\bf i} = (i_1, \dots, i_m) \in I^m$ such that $\sum i_j = d$. Given a
representation $W$ with $\gdim W = d$, a \emph{flag on $W$ of type
  ${\bf i}$} is a flag
\[
W^\bullet = (0 \subset W^1 \subset \dots W^m = W)
\]
of subrepresentations of $W$ such that $W^j / W^{j-1}$ is isomorphic
to the simple representation concentrated at the vertex $i_j$ with all
arrows identically zero, for all
$1 \le j \le m$. Consider the space
\[
E_V ({\bf i}) := \{ (W^\bullet, W) 
\; | \; W \in E_V \text{ and } W^\bullet \text{ is a flag on $W$ of type ${\bf i}$} \}.
\]
There is a natural embedding of $E_V ({\bf i})$ into a product
of $E_V$ and a product of partial flag varieties, and so
$E_V({\bf i})$ has the structure of a complex algebraic
variety. One may define a $G_V$-action on $E_V({\bf i})$
via $g \cdot (W^\bullet, W) := (gW^{\bullet}, g\cdot W)$. Consider the natural projection
\[
\pi_{\bf i} : E_V({\bf i}) \to E_V.
\]
Then $\pi$ is clearly proper and $G_V$-equivariant. It is also not
too difficult to see that $E_V({\bf i},{\bf a})$ is smooth.

\begin{Eg}
  If $\Gamma$ consists of a single vertex $i$ and a single loop and
  ${\bf i} = (i,i,\dots ,i)$ then the image of
  \[\pi_{\bf i} : E_V({ \bf i}) \to E_V\] consists of nilpotent
  endomorphisms and $\pi_{\bf i}$ coincides with the Springer resolution of the
  nilpotent cone in $E_V = \End(V)$.
\end{Eg}

\subsection{Constructible sheaves and KLR algebras} Fix a dimension
vector $d$ and an $I$-graded vector space $V$ with $\gdim V = d$.  In what follows we
wish to consider constructible sheaves on $G_V \setminus
E_V$. However, as this space is usually poorly behaved, we instead
consider the 
$G_V$-equivariant geometry of the space $E_V$. Alternatively we could (and
probably should) consider the quotient stack $[G_V \setminus E_V]$  as in the elegant
treatment of Rouquier \cite{R}. The difference is
essentially aesthetic.

Let us fix a commutative ring of coefficients $\Bbbk$ and consider $D_d :=
D^b_{G_V}(E_V; \Bbbk)$ the $G_V$-equivariant bounded  constructible derived category of
$E_V$ with coefficients in $\Bbbk$ \cite{BL}.

Set $\Seq(d) := \{ (i_1, \dots, i_\Bbbk) \in I^k \; | \;
\sum i_j = d\}$. Given any ${\bf i} \in \Seq(d)$ we can consider the
proper map
\[
\pi_{\bf i} : E_V({\bf i}) \to E_V.
\]
defined in the previous subsection. The direct image
$\LC_{{\bf i}, \Bbbk} := \pi_{{\bf i}!}\un{\Bbbk}$ of the equivariant constant
sheaf on $E_V({\bf i})$ in $D^b_{G_V}(E_V({\bf i}), \Bbbk)$ is
called a \emph{Lusztig sheaf}. We set
\[
\LC_{V,\Bbbk} := \bigoplus_{{\bf i} \in \Seq(d)} \LC_{{\bf i}, \Bbbk} \in D_d.
\]

To $d$ we may also associate a Khovanov-Lauda-Rouquier algebra
$R(d)$ (see \cite{KL,R}). It is a graded algebra with idempotents $e(
{\bf i} )$ corresponding to each ${\bf
  i} \in \Seq(d)$. It is free over $\ZM$ and we denote by $R(d)_\Bbbk$ the
algebra obtained by extension of scalars to our ring $\Bbbk$.

The following result explains the relation between
Khovanov-Lauda-Rouquier algebras and the geometry of the moduli space
of quiver representations:

\begin{thm} \label{thm:KLRiso}
One has an isomorphism of graded rings
\[
R(d)_\Bbbk \cong \Ext^\bullet_{D_d}(\LC_{V,\Bbbk}).
\]
Under this isomorphism $e({ \bf i})$ is mapped to the
projection to $\LC_{{\bf i}, \Bbbk}$.
\end{thm}

\begin{proof}
If $\Bbbk$ is a field of characteristic zero this is proved in
\cite{VV} and \cite[\S~5]{R}. For an arbitrary ring $\Bbbk$ this is
\cite[Theorem 2.5]{M}.
\end{proof}

\subsection{Lusztig sheaves and projective modules}
\label{sec:luszt-sheav-proj}

It is well-known and easily proved that given an object $X$ in a
Karoubian additive category then $\Hom(-,X)$ gives an
equivalence 
\[
\langle X \rangle_{\oplus}^{op} \simto \proj E
\]
where $E := \End(X)$, $\langle X \rangle^{op}_\oplus$ denotes the
opposite category of the full additive Karoubian
subcategory generated by $X$, and $\proj$ denotes the category of
finitely generated projective modules over $E$ (see e.g. \cite[\S
~1.5]{Kr}).

One can extend this observation to graded categories (i.e. categories
equipped with a self-equivalence $M \mapsto TM$). If $X$ is an object
in a graded Karoubian additive category then $\bigoplus_{m \in
  \ZM} \Hom(-, T^mX)$ gives an equivalence of graded categories
\[
\langle X \rangle_{T, \oplus}^{op} \simto \Proj E
\]
where $\langle X \rangle^{op}_{T, \oplus}$ denotes the
opposite category of the full additive Karoubian
subcategory generated by $T^mX$ for all $m \in \ZM$, and $E$ denotes
$\bigoplus_{m \in \ZM} \Hom(X, T^mX)$, naturally a graded
algebra\footnote{given $f : X \to T^mX$ and $g : X \to T^nX$ their product is
given by $T^nf \circ g : X \to T^{m+n}X$.} and $\Proj E$ denotes the
graded category of finitely generated projective modules over $E$
(viewed as a graded category with self-equivalence $M \mapsto M(1)$). 

Let us forget the triangulated structure on $D_d$ and view it simply
as a graded additive category, with self-equivalence given by $\FC
\mapsto \FC[1]$. If we apply the above observations together with Theorem
\ref{thm:KLRiso} (with $V$ and $d$ as in the previous section) we see
that $\Hom^\bullet(-, \LC_{V,\Bbbk})$ yields an equivalence
\[
\langle \LC_{V,\Bbbk} \rangle^{op}_{[1], \oplus} \simto \Proj R(d)_{\Bbbk}.
\]
Moreover, it follows from Maksimau's proof of Theorem \ref{thm:KLRiso}
that the above equivalence is compatible with extension of
scalars.

\subsection{Moduli of representations of Dynkin quivers}
\label{sec:moduli-repr-dynk}
We say that $Q$ is a \emph{Dynkin quiver} if the graph underlying $Q$ is a
simply-laced Dynkin diagram. If $Q$ is a Dynkin quiver we identify $I$ with
the simple roots of the corresponding simply laced root system and
write $R^+ \subset \NM [I]$ for the positive roots. Recall Gabriel's
theorem (see e.g. \cite{B}): a quiver has finitely many indecomposable
representations if and only if the underlying graph of $Q$ is a
Dynkin diagram, in which case we have a bijection:
\[
\gdim: \left \{ \begin{array}{c} \text{indececomposable}\\ \text{representations of
      $Q$} \end{array} \right \}_{/ \cong}
\simto R^+.
\]
Given a positive root $\a \in R^+$ we denote by $I_\a$ the
corresponding indecomposable representation, which is well-defined up
to isomorphism.

Recall that orbits of $G_V$ on $E_V$ correspond to isomorphism
classes of representations of $Q$. Hence $G_V$ has finitely many
orbits on $E_V$ if and only if there are finitely many
isomorphism classes of representations of $Q$ with dimension vector
$d$. By Gabriel's theorem, this is the case for all dimension vectors if
and only if $\Gamma$ is a Dynkin quiver.

From now on let us assume that $Q$ is a Dynkin quiver and fix a
dimension vector $d$. By
Gabriel's theorem and the Krull-Schmidt theorem we see that
$G_V$-orbits on $E_V$ are classified by tuples:
\[
\Lambda_d := \{ \l =  (\l_\a)_{\a \in R^+} \; | \; \sum \l_\a \a = d \}
\]
($\l_\a$ gives the multiplicity of the indecomposable representation
$I_\a$ as a direct summand of the isomorphism class of
representation). Rephrasing this we have a stratification of $E_V$ by $G_V$-orbits
\[
E_V = \bigsqcup_{\l \in \Lambda_d } X_\l
\]
where
$X_\l := \{ W \in E_V \; | \; W
\cong \bigoplus_{\a \in R^+} I_\a^{\oplus \l_\a} \text{ in $\Rep Q$} \}.$

\subsection{Parity sheaves on quiver moduli}
\label{sec:parity-sheav-quiv}

We now recall the notion of a parity complex and sheaf 
\cite{JMW}. As above we assume that $Q$ is a Dynkin quiver and fix the stratification
\[
E_V = \bigsqcup_{\l \in \L_d} X_\l
\]
of $E_V$ by $G_V$-orbits. Given any $\l \in \Lambda_d$ we denote
by $i_\l : X_\l \into E_V$ the inclusion of $X_\l$. Let $? \in \{
!, * \}$. We say that a complex  $\FC \in D_d$ is \emph{$?$-even}
if the cohomology sheaves of $i_\l^? \FC$ are local systems of free
$\Bbbk$-modules which vanish in odd degrees. We say that a complex
$\FC \in D_d$ is \emph{even} if it is both $*$- and $!$-even. We say
that $\FC$ is \emph{parity} if it admits a decomposition
$\FC \cong \FC_0 \oplus \FC_1$ with both $\FC_0$ and $\FC[1]$
even. A complex $\FC$ is a \emph{parity sheaf} if it is parity,
self-dual and indecomposable.

Now suppose that $\Bbbk$ is a complete local ring, so that $D_d$ is a
Krull-Schmidt category (see \cite[\S~2.1]{JMW}). We have the following
classification result:

\begin{thm}
  For all $\l \in \L_d$ there exists up to isomorphism at most one
  parity sheaf $\EC(\l, \Bbbk)$ with $\supp \EC(\l, \Bbbk) = \ov{X_\l}$.
\end{thm}

\begin{proof}
  The theorem is immediate from the theory of parity sheaves, once
  we have established that our stratified $G_V$-variety $E_V$
  satisfies \cite[(2.1) and (2.2)]{JMW} and that all $G_V$-equivariant
  local systems on $X_\l$ are constant. Both these statements hold by 
\cite[Corollary 2.11]{M} and \cite[Lemma 3.6]{M}.\end{proof}

For a general Dynkin quiver one does not know if Lusztig sheaves are
parity (see \cite[Conjecture 1.3]{M}). The problem is that one does
not 
know of the fibres of the maps $\pi_{\bf i}$ have vanishing odd
cohomology. If $Q$ is of type $A$ then this has been established by
Maksimau (see \cite[Corollary 3.36]{M}):

\begin{thm}  \label{thm:Maks}
Suppose that $Q$ is of type $A$. Then for all $\l \in
  \Lambda_d$ there exists a parity sheaf $\EC(\l, \Bbbk)$ with $\supp \EC(\l, \Bbbk)
  = \ov{X_\l}$. Moreover, $\LC_{V,\Bbbk}$ is parity and for all $\l \in \L_d$,
  a shift of $\EC(\l, \Bbbk)$ occurs as a direct summand of $\LC_{V,\Bbbk}$. 
\end{thm}

\begin{remark}
  Actually, Maksimau establishes the existence of parity sheaves for
  any Dynkin quiver. Surprisingly, it seems difficult to show that
  they occur as summands of Lusztig sheaves.
\end{remark}

\subsection{Parity sheaves and the Kleshchev-Ram conjecture}
\label{sec:parity-sheav-kleshch}

In this
section we assume that $Q$ is a Dynkin quiver of type $A$. We fix a
dimension vector $d$. Recall our $p$-modular system $(\KM, 
\OM, \FM)$ from \S \ref{sec:gam}.

Recall that the KLR algebra $R(d)$ is a graded ring, with each graded
component free and finite rank over $\ZM$. Moreover $R(d)$ contains a
polynomial subring and the symmetric polynomials yield a subring $A
\subset R(d)$ in the centre of $R(d)$. In fact, $A$ is equal to the
centre of $R(d)$ \cite[Theorem 2.9]{KL}. Finally, $R(d)$ is free of finite rank over
$A$ by \cite[Corollary 2.10]{KL}. Hence if we set $H_\OM := R(d)
\supset A_\OM$ then $H_\OM$ satisfies the conditions of \S
\ref{sec:gam}.

On the other hand, if we apply the observations in
\S\ref{sec:luszt-sheav-proj} for $\Bbbk \in \{ \KM, \OM, \FM \}$ we
obtain an equivalence of graded categories
\begin{equation} \label{eq:parityKLR}
\langle \LC_{V,\Bbbk} \rangle_{[1],\oplus}^{op} \simto \Proj R(d)_{\Bbbk}.
\end{equation}
By the Theorem \ref{thm:Maks} of Maksimau the indecomposable summands
of $\LC_{V,\Bbbk}$ coincide (up to shifts) with the parity sheaves. We
conclude that the indecomposable graded projective modules (and hence
also the graded simple modules) are parametrised up to shift by
$\Lambda_d$. 

\begin{remark}
  This fact has been established algebraically for any Dynkin quiver
  by Kleshchev and Ram \cite{KR}.
\end{remark}

The following is the main result of this section:

\begin{thm} \label{thm:topologicalKR}
  The following are equivalent:
  \begin{enumerate}
  \item The Kleshchev-Ram conjecture holds for $R(d)$: the
    decomposition map \[[\Rep^f R(d)_{\KM}] \to [\Rep^f
    R(d)_{\FM}]\] is trivial.
\item For all $\l \in \Lambda_d$ the parity complex $\EC(\l, \OM)
  \otimes^L_{\ZM_p} \KM$ is indecomposable.
\item For all $\l \in \Lambda_d$ the stalks and costalks of the
  intersection cohomology complex
  $\ic(\overline{X_\l}, \OM)$ are free of $p$-torsion.
  \end{enumerate}
\end{thm}

\begin{proof}
  By Brauer reciprocity (Lemma \ref{lem:BR}) statement (1) is
  equivalent to the extension map
\[
e : [\Proj R(d)_\FM] \to [\Proj R(d)_\KM]
\]
being the identity. Now the equivalence of (1) and (2) follows from
the fact that \eqref{eq:parityKLR} is compatible with extension of
scalars.

It remains to show the equivalence of (2) and (3). 

\emph{(2) $\Rightarrow$ (3):} If $\Bbbk = \KM$
then we can apply the decomposition theorem \cite{BBD} to conclude the $\LC_{V,\Bbbk}$ is
isomorphic to a direct sum of shifts of intersection cohomology
complexes. We conclude that $\EC(\l,\KM) \cong \ic(\overline{X_\l},
\KM)$. By the uniqueness of parity sheaves, $\EC(\l, \OM)
\otimes_{\OM} \KM$ is indecomposable if and only if $\EC(\l, \OM)
\otimes_{\OM} \KM \cong \ic(\overline{X_\l}, \KM)$.

Let $\EC$ be a complex of sheaves of $\OM$-modules with
torsion free stalks and costalks. Then $\EC \otimes_\OM \KM \cong
\ic(\overline{X_\l}, \KM)$ if and only if $\EC \cong
\ic(\overline{X_\l}, \OM)$, as follows from the
characterisation of $\ic(\overline{X_\l},\OM)$ in terms of stalks and
costalks \cite[Proposition 2.1.9 and \S~3.3]{BBD}.

Putting these two observations together, we conclude that
$\EC(\l, \OM)\otimes_{\OM} \KM$ is indecomposable if and
only if $\EC(\l,\OM) \cong \ic(\overline{X_\l}, \OM)$. Hence (2)
implies (3), because the stalks and costalks of $\EC(\l,\OM)$ are free
of $p$-torsion by definition.

\emph{(3) $\Rightarrow$ (2)}: If $\ic(\overline{X_\l},\OM)$ has torsion free
stalks and costalks then it is parity (because $\ic(\overline{X_\l},
\OM) \otimes_\OM \KM \cong \ic(\overline{X_\l}, \KM)$ is). Hence
$\ic(\overline{X_\l},\OM) \cong \EC(\l, \OM)$ and 
$\EC(\l, \OM) \otimes_{\OM} \KM \cong \ic(\overline{X_\l},\OM) \otimes_{\OM}
\KM$ is indecomposable. Hence (3) implies (2).
\end{proof}

\begin{remark}
  The advantage of condition (3) in the above theorem is that it is a
  purely topological condition, and hence we can use it to move the
  Kleshchev-Ram conjecture to a question about equivalent
  singularities in the flag variety, where more is known and calculations are easier.
\end{remark}

\begin{remark}
  See \cite[Proposition 3.11]{WB} for results along similar lines to the above theorem.
\end{remark}

\section{Counterexamples}

In this section we assemble some known results due to Braden and Polo
and use Theorem \ref{thm:topologicalKR} to give counterexamples to the
Kleshchev-Ram conjecture for quivers of type $A$.

\subsection{Quiver varieties and flag varieties} We briefly recall a
construction (which I learnt from \cite[\S 8.1]{KS}) which relates an open
subvariety inside the moduli of representations of a type $A$ quiver
to singularities of Schubert varieties in a flag variety of type $A$.

Fix $n \ge 0$ and consider the quiver $Q$ of type $A_{2n-1}$ with
the following orientation:
\[
1 \to 2 \to \dots \to n \from n+1 \from \dots \from 2n-1.
\]
Consider the dimension vector $d = (1, 2, \dots, n, \dots, 2, 1)$. Let
$I$ denote the vertices of $Q$ and let $V$ denote an $I$-graded vector
space with dimension vector $d$. Inside $E_V$ consider the
open subvariety $U$ consisting of representations all of whose arrows are
injective. The product of the automorphisms of $V_i$ for $i \ne n$ act
freely on $U$ and the quotient is isomorphic to the space of pairs on
flags on $V_n \cong \CM^n$, with the natural action of $G =
GL(V_n)$.

Hence, after fixing a Borel subgroup $B \subset G$, the
singularities of the closures of $G_V$-orbits on $U$ are equivalent to
the singularities of closures of $G$ orbits on $G/B \times G/B$, or
equivalently to the singularities of Schubert varieties in
$G/B$. Combining these observations with Theorem
\ref{thm:topologicalKR} we deduce that if the Kleshchev-Ram conjecture
holds then stalks and costalks of all intersection cohomology
complexes of Schubert varieties in $G/B$ are torsion free. The first counterexamples
to this statement were given by Braden in 2004 (see the appendix to \cite{WB}). He gave
examples of $2$-torsion in the costalks of intersection cohomology
complexes on the flag variety of $GL_8$.

For several years since Braden's announcement of his results the
existence of $p \ne 2$ torsion was not 
known. Recently Polo \cite{P} has proved that for all prime numbers $p$
there exists a Schubert variety in the flag variety of $GL_{4p}(\CM)$ whose intersection
cohomology complex has $p$-torsion in
its costalk. It then follows from Theorem \ref{thm:topologicalKR} that
the Kleshchev-Ram conjecture is false for all primes.\footnote{The
  author has recently found even more torsion in flag varieties of
  type $A$ \cite{W}.}

The modules involved in the above counterexamples are enormous
(the first counterexample involves a quiver of type $A_{15}$!). It
seems unlikely that one could verify these 
counterexamples algebraically, even with the help of a powerful
computer. The rest of this paper is devoted to the description of the
Kashiwara-Saito singularity, which occurs in a quiver of type
$A_5$, and is small enough that one (i.e. Jon Brundan) can verify
algebraically that it provides a counterexample.

\subsection{The Kashiwara-Saito singularity}
Let $Q$ denote the $A_5$ quiver:
\[
Q = 1 \to 2 \to 3 \to 4 \to 5.
\]
We identify the vertices $I$ of $Q$ with the simple roots $\a_1,
\dots, \a_5$ of a root system $R$ of type $A_5$. For $1 \le i \le j \le
5$ let $\a_{ij} := \a_i + \dots + \a_j$. Then the positive roots of
$R$ are
$\{ \alpha_{ij} \; | \; 1 \le i \le j \le 5
\}$. Consider the dimension vector $d = 2 \a_{1} + 4 \a_{2} + 4\a_{3} +
4\a_4 + 2 \a_5$ and let $V$ be an $I$-graded vector space with $\gdim
V = d$.

As described in \S~\ref{sec:moduli-repr-dynk}, $G_V$ orbits on $E_V$
are parametrised by the set
\[
\Lambda_d := \left \{\sum_{1 \le i \le j \le 5} \lambda_{ij}\alpha_{ij} \; | \; \sum \lambda_{ij}\alpha_{ij} = d \right \}.
\]
Consider
\begin{gather*}
\sigma := \a_{12} + \a_{23} + \a_{34} + \a_{45} + \a_{14} + \a_{25}\\
\pi := 2(\a_{33} + \a_{12} + \a_{45} + \a_{24}).
\end{gather*}
then $\pi, \sigma \in \Lambda_d$ and we let $X_\pi$ and $X_\sigma$
denote the corresponding $G_V$-orbits on $E_V$.

\begin{prop}
  $\ic(\overline{X_\sigma}, \ZM)$ has 2-torsion in its costalk at any
  point of $X_\pi$.
\end{prop}

\begin{remark}
It is a result due to Kashiwara and
Saito (see \cite[Theorem 7.2.1]{KS} and the subsequent remark) that the singular support of the
intersection cohomology $D$-module on $\overline{X_\pi}$ is equal to
the union of the closures of the conormal bundles to $X_\pi$ and $X_\sigma$. In
particular it is reducible. Braden \cite{B2} confirms that the characteristic cycle is equal to the sum of
the fundamental classes.
\end{remark}

\begin{proof} Let $M_n(\CM)$ denote the space of $n \times n$-complex matrices over
$\CM$. Consider the space $S$ of matrices $M_i \in M_2(\CM)$ for $i \in \ZM/4\ZM$
satisfying the two conditions
\begin{gather}
  \label{eq:1}
\rank M_i \le 1 \text{ for $i \in \ZM/4\ZM$,} \\
M_iM_{i+1} = 0 \text{ for $i \in \ZM/4\ZM$}.
\end{gather}
Clearly $S$ is an affine variety. One can show that its dimension is
8. We call $S$ (or more precisely the singularity of $S$ at $0 \in S$)
the \emph{Kashiwara-Saito} singularity. It is known \cite[Lemma
2.2.2]{KS} that the singularity of $\overline{X_\sigma}$ at $X_\pi$
is smoothly equivalent to the singularity of $S$ at $0$.

Let $G = GL_8(\CM)$ and $B$ denote the subgroup of upper triangular
matrices, and identify the Weyl group of $G$ with the permutation
matrices. Consider the Schubert variety $X_x = \overline{BxB/B}$ where
$x$ is the permutation $x = 62845173$. Then it is not difficult to
check that the singularity of $X_x$ along the Schubert
variety $ByB/B$ where $y = 21654387$ is smoothly equivalent the
singularity of $S$ at 0 (see \cite[Example 8.3.1]{KS}).

Because the stalks and costalks of the intersection cohomology
complexes are invariant (up to a shift) under smooth equivalence, the
proposition follows once one knows that $\ic(X_x, \ZM)$ has
2-torsion in its costalk at $y$. This has been
verified by Braden (using similar techniques to those used in the
appendix to
\cite{WB}), by the author (using generators and
relations for Soergel bimodules \cite{EW}) and by Polo \cite{P} (his
examples of torsion for all primes $p$ gives this example for $p = 2$).
\end{proof}

By Theorem \ref{thm:topologicalKR} we conclude that the projective
module $P(\sigma,\ZM_2)$ for $R(d)_{\ZM_2}$ corresponding to
$\EC(\sigma,\ZM_2)$ is decomposable when we tensor with $\QM_2$. Hence
the Kleshchev-Ram conjecture fails for $R(d)$. One may show that
\[
\EC(\sigma,\ZM_2) \otimes_{\ZM_2} \QM_2 \cong \ic(\overline{X_\sigma},
\QM_2) \oplus \ic(\overline{X_\pi},\QM_2).
\]
By Brauer reciprocity it follows that the simple module for $R(d)$
indexed by $\pi$ becomes reducible when one reduces modulo $2$. In
fact, in $[\Rep R(d)_{\FM_2}]$ one has
\[
[L(\pi, \ZM_2)\otimes_{\ZM_2} \FM_2] = [L(\pi, \FM_2)] + [L(\sigma,\FM_2)].
\]
where $L(\pi,\ZM_2)$ denotes an integral form of the
$R(d)_{\QM_2}$-module labelled by $\pi$, and $L(\pi,\FM_2)$ and
$L(\sigma,\FM_2))$ denote the simple $[\Rep R(d)_{\FM_2}]$-modules
labelled by $\sigma$ and $\pi$ respectively.
This has been verified by direct algebraic computations by Jon Brundan
(aided by his computer) and Alexander Kleshchev (with his bare
hands) and has recently become available \cite[2.6]{BK2}.

\end{document}